\documentclass[10pt]{amsart}
\usepackage[all]{xy}
\usepackage{amsmath}
\usepackage{amssymb}
\usepackage{amsthm}
\usepackage{graphicx}
\usepackage{blkarray}
\usepackage{float}
\usepackage{enumerate}
\usepackage{xcolor}
\usepackage{comment}
\usepackage{tikz}
\usepackage{nccmath}
\usepackage{hyperref}

\newtheorem{theorem}{Theorem}[section] 

\newtheorem{corollary}[theorem]{Corollary}
\newtheorem{proposition}[theorem]{Proposition}

\theoremstyle{definition}
\newtheorem{definition}[theorem]{Definition}
\newtheorem{example}[theorem]{Example}
\newtheorem{remark}[theorem]{Remark}

\numberwithin{equation}{section}

\begin{document}

\title[Square integrability of regular representations]{Square integrability of regular representations on reductive homogeneous spaces}

\author{Kazushi Maeda}
\address[K.Maeda]{Graduate School of Mathematical Sciences, The University of Tokyo, 3-8-1 Komaba, Meguro, 153-8914 Tokyo, Japan}
\email{kmaeda@ms.u-tokyo.ac.jp}

\author{Yoshiki Oshima}
\address[Y.Oshima]{Graduate School of Mathematical Sciences, The University of Tokyo, 3-8-1 Komaba, Meguro, 153-8914 Tokyo, Japan}
\email{yoshima@ms.u-tokyo.ac.jp}



\begin{abstract}
Let $G$ be a real reductive Lie group and $H$ a reductive subgroup of $G$. 
Benoist-Kobayashi studied when $L^2(G/H)$ is a tempered representation of $G$
 and in particular they gave a necessary and sufficient condition for the temperedness
 in terms of certain functions on Lie algebras.
In this paper, we consider when $L^2(G/H)$
 is equivalent to a unitary subrepresentation of $L^2(G)$
 and we will give a sufficient condition for this
 in terms of functions introduced by Benoist-Kobayashi.
As a corollary, we prove the non-existence of discrete series for homogeneous spaces $G/H$
 satisfying certain conditions. 
\end{abstract}

\maketitle

\section{Introduction}
Let $G$ be a real algebraic reductive Lie group and $H$ an algebraic reductive subgroup.
The quotient space $G/H$ has a $G$-invariant measure and 
 we have a unitary representation $L^2(G/H)$ of $G$,
 which is called the regular representation on $G/H$.

When $G/H$ is a symmetric space, 
 a great deal of studies has been done on the irreducible decomposition of $L^2(G/H)$
 and Plancherel formulas are known in this case.
On the other hand, such formulas are widely open for general homogeneous spaces $G/H$.

In a series of papers \cite{BK, BKII, BKIII, BKIV}, Benoist and Kobayashi studied
 the temperedness of $L^2(G/H)$. 
In the first paper \cite{BK},
 they gave a criterion on a pair of real reductive groups $(G, H)$ 
 for $L^2(G/H)$ to be tempered.
Let $\mathfrak{g}$ and $\mathfrak{h}$ be Lie algebras of $G$ and $H$,
 respectively.
For an $\mathfrak{h}$-module $V$, 
 they defined a non-negative valued piecewise linear function $\rho_V$
 on a maximal split abelian subspace $\mathfrak{a}$ of $\mathfrak{h}$.
Then they proved that $L^2(G/H)$ is tempered if and only if
 the inequality $\rho_{\mathfrak{g}}\leq 2\rho_{\mathfrak{g}/\mathfrak{h}}$
 holds on $\mathfrak{a}$.
A key ingredient of their proof is the estimate
 of the volume of $gC\cap C$ for a compact set $C\subset G/H$
 when $g\in G$ goes to infinity.

In this paper, we consider when $L^2(G/H)$ is equivalent to 
 a unitary subrepresentation of the left regular representation $L^2(G)$.
We write $L^2(G/H)\subset L^2(G)$ for this property.
This condition is stronger than the temperedness of $L^2(G/H)$.
Our main result (Theorem~\ref{thm:SqintRho}) is
 $L^2(G/H)\subset L^2(G)$
 if the strictly inequality $\rho_{\mathfrak{g}}< 2\rho_{\mathfrak{g}/\mathfrak{h}}$
 holds on $\mathfrak{a}\setminus\{0\}$.
The proof is based on the estimate of the volume of $gC\cap C$ obtained in 
 Benoist-Kobayashi~\cite{BK}.

Broadly speaking, these results say that $H$ is small compared with $G$,
 then the representation $L^2(G/H)$ gets close to $L^2(G)$.
The temperedness of $L^2(G/H)$ means that $L^2(G/H)$ can be disintegrated by
 irreducible tempered representations.
Hence any irreducible unitary representation which contributes to 
 the decomposition of $L^2(G/H)$ also contributes to $L^2(G)$.
On the other hand, the existence of the inclusion $L^2(G/H)\subset L^2(G)$
 determines how each irreducible representation appears in $L^2(G/H)$.
For example, if an irreducible unitary representation 
 appears in a discrete spectrum of $L^2(G/H)$,
 then it also appears in a discrete spectrum of $L^2(G)$. 
In particular, we obtain the non-existence of the discrete series for $G/H$
 if $G$ does not have a discrete series and $L^2(G/H)\subset L^2(G)$.

\section{Square Integrable Representations}

Let $G$ be a unimodular Lie group.
For an irreducible unitary representation $\pi$ of $G$,
 the following three conditions are known to be equivalent
 (see Howe-Tan~\cite[Chapter V, Proposition 1.2.3]{HT}):
\begin{enumerate}
\item 
There exists a nonzero matrix coefficient
 of $\pi$ which is a square integrable function on $G$.  
\item 
All matrix coefficients of $\pi$ are square integrable.
\item
$\pi$ is unitarily equivalent to a subrepresentation of
 the left regular representation $L^2(G)$.
\end{enumerate}
Such $\pi$ is called a square integrable representation.
If $G$ is semisimple, it is also called
 Harish-Chandra's discrete series representation.

We would like to generalize this notion to possibly reducible
 unitary representations of $G$. 

\begin{proposition}\label{prop:sqint}
Let $G$ be a unimodular Lie group 
and let $(\pi,\mathcal{H})$ be a unitary representation of $G$.
The following conditions are equivalent.
\begin{enumerate}
\item There exists a dense subset $V\subset \mathcal{H}$
 such that for any vectors $u,v\in V$
 the matrix coefficient $c_{u,v}(g)=\langle\pi(g)u,v\rangle$
 is a square integrable function on $G$.
\item $(\pi,\mathcal{H})$ is unitarily equivalent to a subrepresentation 
 of a direct sum of copies of the left regular representation $L^2(G)$.
\end{enumerate}
\end{proposition}

\begin{proof}
The implication (1) $\Rightarrow$
 (2) is proved in \cite[Chapter V, Corollary 1.2.4]{HT}. 

To prove the converse direction, 
 suppose that $\mathcal{H}$ is a subrepresentation of $L^2(G)^{\oplus I}$
 for a set $I$.
Let 
\[\widetilde{V}:=\{(f_i)_{i\in I}\in L^2(G)^{\oplus I}
\mid f_i\in C_c(G) \text{ for all $i\in I$ and  $f_i=0$
 for all but finite $i\in I$}\}.\]
Here, $C_c(G)$ denotes the space of compact supported continuous functions.
Then $\widetilde{V}$ is a dense subspace of $L^2(G)^{\oplus I}$.
Let $p\colon L^2(G)^{\oplus I}\to \mathcal{H}$ be the orthogonal projection
 and $V:=p(\widetilde{V})$ so that 
 $V$ is a dense subspace of $\mathcal{H}$.
We will show that (1) holds true for this $V$, namely,
 show that $G\ni g\mapsto\langle\pi(g)(u),v\rangle_{\mathcal{H}}$
 is square integrable for $u, v\in V$.
Take $\tilde{v}\in \widetilde{V}$ such that $p(\tilde{v})=v$.
We may assume $\tilde{v}=(v_i)_{i\in I}$ such that
 $v_i\neq 0$ for only one $i\in I$.
Then $\langle\pi(g)(u),v\rangle_{\mathcal{H}}
 = \langle\lambda(g)(u),\tilde{v}\rangle_{L^2(G)^{\oplus I}}
 = \langle\lambda(g)(u_i),v_i\rangle_{L^2(G)}$,
 where $u=(u_i)_{i\in I}$ and
 $\lambda$ denotes the left regular representation of $G$.
Therefore, it is enough to show that
 $G\ni g\mapsto\langle\lambda(g)(u),v\rangle_{L^2(G)}$
 is a square integrable function on $G$
 for any $u\in L^2(G)$ and $v\in C_c(G)$.
This can be verified by the following calculation:
\begin{align*}
&\int_{G} \bigl|\langle\lambda(g)(u),v\rangle_{L^2(G)}\bigr|^2 dg
=\int_{G} \Bigl| \int_{G} u(g^{-1}x)\overline{v(x)} dx \Bigr|^2 dg \\
&\leq \int_{G} \Bigl(\int_{G} |u(g^{-1}x)v(x)| dx \cdot
 \int_{G} |u(g^{-1}y)v(y)| dy \Bigr)dg\\
&= \int_{G}\int_{G} 
 \Bigl(\int_{G} |u(g^{-1}x)||u(g^{-1}y)|dg\Bigr)|v(x)v(y)|dxdy\\
&\leq \|u\|_{L^2(G)}^2\|v\|_{L^1(G)}^2 <\infty. \qedhere
\end{align*}
\end{proof}

\begin{definition}\label{de:sqint}
We say a unitary representation $\pi$ is \emph{square integrable}
 if it satisfies the conditions in Proposition~\ref{prop:sqint}. 
\end{definition}

A closely related notion to the square integrability
 is the temperedness of representations.
The temperedness of possibly reducible unitary representation was
 considered in Benoist-Kobayashi~\cite{BK}. 
Here, we recall definitions following \cite{Kaz, BK}. 

\begin{definition}
Let $(\pi,\mathcal{H})$ and $(\rho,\mathcal{K})$
 be two unitary representations of a Lie group $G$.
We say $\pi$ is weakly contained in $\rho$
 if for every $u\in \mathcal{H}$,
 every compact subset $Q$ of $G$, and every $\varepsilon >0$,
 there exist $v_i\ (1\leq i\leq n)$ in $\mathcal{K}$ such that 
  \[\Bigl|\ \langle \pi(g)u, u \rangle - 
  \sum_{i = 1}^n \langle \rho(g) v_i, v_i\rangle \ \Bigr| < \varepsilon\]
  for all $g \in Q$.
\end{definition}

\begin{definition}\label{de:temp}
A unitary representation $\pi$ of a Lie group $G$ is said to be \emph{tempered} if $\pi$ is weakly contained in the left regular representation $L^2(G)$.
\end{definition}

\begin{proposition}[{\cite[Theorems 1, 2 and Corollary]{CHH}}]
\label{prop:temp}
  Let $G$ be a semisimple Lie group with finite center, $\pi$ a unitary representation of $G$. Then $\pi$ is tempered if and only if 
 there exists a dense subset $V$ of $\mathcal{H}$ such that
 the function
 $g\mapsto \langle\pi(g)u, v\rangle$
 belongs to $L^{2+\varepsilon}(G)$ for any $u,v \in V$
 and $\varepsilon > 0$.
\end{proposition}

It is easy to see from Definition~\ref{de:temp}
 that any square integrable unitary representation
is tempered.

Let $\widehat{G}_{\textup{temp}}$ be the set of equivalence classes of
 irreducible tempered representations of a real reductive Lie group $G$.
Then Harish-Chandra's Plancherel formula of $L^2(G)$ is given as 
\begin{equation}\label{eq:PlL2G}
L^2(G)\simeq \int_{\widehat{G}_{\textup{temp}}}^{\oplus}
 \mathcal{H}_{\sigma}\otimes \mathcal{H}_{\sigma}^* \: d\mu(\sigma).
\end{equation}
Suppose that $(\pi, \mathcal{H})$ is a unitary representation of $G$
 such that $\mathcal{H}$ is a separable Hilbert space.
Then $\pi$ is tempered if and only if $\pi$ is disintegrated into
 irreducible tempered representations (see \cite[Remark 2.6]{BK}):
\begin{equation*}
\pi \simeq \int_{\widehat{G}_{\textup{temp}}}^{\oplus}
 \mathcal{H}_{\sigma}^{\oplus m(\sigma)} \: d\nu(\sigma),
\end{equation*}
where $m(\sigma)$ is the multiplicity function 
 $\widehat{G}_{\textup{temp}}\to \mathbb{N}\cup\{\infty\}$.
In terms of the irreducible decomposition,
 $\pi$ is square integrable if and only if one can take
 the measure $\nu$ to be equivalent to $\mu$.

If we assume moreover that the semisimple part of $G$ is noncompact,
 then $\mathcal{H}_{\sigma}$ is infinite-dimensional for all
 $\sigma\in \widehat{G}_{\textup{temp}}$.
Therefore, the multiplicities of $\sigma\in \widehat{G}_{\textup{temp}}$
 in the left regular representation $L^2(G)$ are all infinite
 by \eqref{eq:PlL2G},
 which implies that $L^2(G)^{\oplus \mathbb{N}}$ is equivalent to $L^2(G)$
 as a unitary representation of $G$.
We thus obtain the following proposition.

\begin{proposition}\label{prop:L2sub}
Suppose that $G$ is a reductive Lie group and its semisimple part is noncompact.
Let $(\pi, \mathcal{H})$ be a unitary representation of $G$
 such that $\mathcal{H}$ is a separable Hilbert space.
Then $\pi$ is square integrable if and only if $\pi$ is unitarily equivalent to
 a subrepresentation of the left regular representation $L^2(G)$.
\end{proposition}

\section{Square integrability of $L^2(G/H)$}\label{sec:SqintReg}
Let $G$ be an algebraic reductive group and $H$ an algebraic reductive subgroup of $G$.
In this section we study when $L^2(G/H)$ is a square integrable representation of $G$.

For this purpose, we recall the piecewise linear function $\rho_V$ on a Lie algebra introduced in
 Benoist-Kobayashi~\cite{BK}.

Let $\mathfrak{h}$ be a real reductive Lie algebra
 and let $\mathfrak{a}$ be a maximal split abelian subspace of $\mathfrak{h}$.
For a finite-dimensional $\mathfrak{h}$-module $(\pi, V)$, we define a function
 $\rho_V \colon \mathfrak{a} \rightarrow \mathbb{R}_{\geq 0}$ by
\[\rho_V(Y) := \displaystyle \frac{1}{2} \sum_{\lambda \in \Lambda_Y} m_{\lambda} \left| \operatorname{Re} \lambda \right|\quad (Y \in \mathfrak{a}),\]
where $\Lambda_Y$ is the set of all eigenvalues of $\pi(Y)$ in the complexification $V_\mathbb{C}$ of $V$ and $m_{\lambda}$ is the multiplicity of the eigenvalue $\lambda$.
By definition, $\rho_V(Y)\geq 0$ for all $Y\in \mathfrak{a}$
 and $\rho_V(Y)=0$ implies that $Y$ acts as zero on $V$.

Benoist-Kobayashi \cite{BK} characterizes the temperedness of $L^2(G/H)$
 in terms of the function $\rho_V$.
Let $\mathfrak{q}:=\mathfrak{g}/\mathfrak{h}$.
By the adjoint action, $\mathfrak{g}$ and $\mathfrak{q}$ can be regarded as 
 $\mathfrak{h}$-modules.

\begin{theorem}[{\cite[Theorem 4.1]{BK}}]\label{thm:TempRho}
 Let $G$ be an algebraic semisimple Lie group and $H$ an algebraic reductive subgroup of $G$. 
Then $L^2(G/H)$ is tempered if and only if
 $\rho_{\mathfrak{g}}(Y) \leq 2\rho_{\mathfrak{q}}(Y)$ for any $Y \in \mathfrak{a}$.
\end{theorem}

The following is a main theorem of this paper,
 which gives a sufficient condition for $L^2(G/H)$ to be square integrable
 in terms of functions $\rho$.

\begin{theorem}\label{thm:SqintRho}
Let $G$ be an algebraic reductive Lie group and
 $H$ an algebraic reductive subgroup of $G$. 
The unitary representation of $G$ in $L^2(G/H)$ is a square integrable representation if 
 $\rho_{\mathfrak{g}}(Y) < 2\rho_{\mathfrak{q}}(Y)$ for any $Y\in \mathfrak{a} \setminus \{0\}$.
\end{theorem}

We note that in the setting of Theorem~\ref{thm:SqintRho}, 
 $L^2(G/H)$ is a square integrable representation if and only if 
 $L^2(G/H)$ is unitarily equivalent to a subrepresentation of $L^2(G)$.
If the semisimple part of $G$ is noncompact, this follows from Proposition~\ref{prop:L2sub}.
If the semisimple part of $G$ is compact, it is easy to see that these conditions are
 equivalent to the compactness of $H$.

For the proof of the theorem, we use the following proposition, 
 the $p=2$ case of \cite[Proposition 4.3]{BK}.

\begin{proposition}\label{prop:VolEstimate}
Let $G$ be an algebraic reductive Lie group and
 $H$ an algebraic reductive subgroup of $G$ such that
 the action of $G$ on $G/H$ has compact kernel.
Then $\mathrm{vol}(gC\cap C)\in L^2(G)$ for any compact set $C$ in $G/H$
 if and only if $\rho_{\mathfrak{g}}(Y) < 2\rho_{\mathfrak{q}}(Y)$ for any $Y\in \mathfrak{a} \setminus \{0\}$. 
\end{proposition}

\begin{proof}[Proof of Theorem~\ref{thm:SqintRho}]
Let $G$ and $H$ be as in the statement of the theorem and suppose
 $\rho_{\mathfrak{g}}(Y) < 2\rho_{\mathfrak{q}}(Y)$
 for any $Y\in \mathfrak{a} \setminus \{0\}$.
Let $S$ denote the kernel of the $G$-action on $G/H$.
We have $S\subset H$ and $S$ is a normal subgroup of $G$. 
Since the adjoint action of $S$ is trivial on $\mathfrak{q}$, 
 the function $\rho_{\mathfrak{q}}$ is zero on
 $\operatorname{Lie}(S)\cap \mathfrak{a}$.
Then our assumption $\rho_{\mathfrak{g}} < 2\rho_{\mathfrak{q}}$
 implies $\operatorname{Lie}(S)\cap \mathfrak{a}=\{0\}$
 and hence $S$ must be compact.
We can therefore apply Proposition~\ref{prop:VolEstimate}
 and conclude that 
 $\mathrm{vol}(gC\cap C)\in L^2(G)$ for any compact set $C$ in $G/H$.
Let $V$ be the linear span of the characteristic functions $1_C$
 for all compact sets $C$ in $G/H$.
Then $V$ is dense in $L^2(G/H)$ and 
 the matrix coefficient $c_{u,v}$ of any $u,v\in V$ belongs to $L^2(G)$.
\end{proof}

\begin{remark}
In contrast to Theorem~\ref{thm:TempRho}, the converse implication
 of Theorem~\ref{thm:SqintRho} is not true in general.
For example, let $(G,H)=(SL(3,\mathbb{R}),SL(2,\mathbb{R}))$,
 where $SL(2,\mathbb{R})$ is the semisimple part of the symmetric subgroup
 $S(GL(2,\mathbb{R})\times GL(1,\mathbb{R}))$.
Then $\rho_{\mathfrak{g}}=2\rho_{\mathfrak{q}}$ on $\mathfrak{a}$
 but $L^2(G/H)$ is square integrable.
\end{remark}

As an application of Theorem~\ref{thm:SqintRho},
 we will obtain the non-existence of discrete series for
 some homogeneous spaces $G/H$.
Let $\textup{Disc}(G/H)$ denote the discrete series for $G/H$,
 namely, $\textup{Disc}(G/H)$ is the set of equivalence classes
 of irreducible unitary representations of $G$ which are unitary subrepresentations of $L^2(G/H)$.
When $H$ is trivial, $\textup{Disc}(G)=\textup{Disc}(G/1)$
 consists of the irreducible square integrable representations.
Harish-Chandra~\cite{HC66}
 proved that $\textup{Disc}(G)\neq \emptyset$ if and only if
 $\operatorname{rank} G=\operatorname{rank} K$, where $K$ is a maximal compact subgroup of $G$.

\begin{corollary}\label{cor:DS}
Let $G$ be an algebraic reductive group and
 $H$ an algebraic reductive subgroup of $G$. 
Suppose that $\rho_{\mathfrak{g}}(Y) < 2\rho_{\mathfrak{q}}(Y)$ for any
 $Y \in \mathfrak{a} \setminus \{0\}$.
Then $\textup{Disc}(G/H) \subset \textup{Disc}(G)$.
In particular, if moreover
 $\textup{Disc}(G) = \emptyset$, then $\textup{Disc}(G/H) = \emptyset$.
\end{corollary}

\section{Examples}

We apply results of Section~\ref{sec:SqintReg}
 to several homogeneous spaces as examples to obtain
 the square integrability and the non-existence of discrete series.
The following assertions can be verified by
 explicit calculations of functions $\rho$, 
 Theorem~\ref{thm:SqintRho} and Corollary~\ref{cor:DS}.

\begin{example}
Let $n_1 \geq n_2 \geq \cdots \geq n_r \geq 1 $ and 
 $n = n_1 + \cdots + n_r$.
Then the unitary representation 
 $L^2(SL(n,\mathbb{R})/\prod_{k=1}^{r}SL(n_k,\mathbb{R}))$
 of $SL(n,\mathbb{R})$ is square integrable
 if $2n_1 \leq n$ and $n_1 + n_2 <n$.
\end{example}

\begin{example}
  Let $n \geq p+q \geq 3$.
  Then $L^2(SL(n,\mathbb{R})/SO(p,q))$ is
 always a square integrable representation of $SL(n,\mathbb{R})$
 and we have $\textup{Disc}(SL(n,\mathbb{R})/SO(p,q))=\emptyset$.
\end{example}

\begin{example}\label{ex:SOpq}
Let $p \geq p_1 + \cdots + p_r$ and $q \geq q_1 + \cdots + q_r$.
Then the unitary representation
 $L^2(SO(p,q)/\prod_{k=1}^{r}SO(p_k,q_k))$ of $SO(p,q)$ 
 is square integrable 
 if $2\max_{p_k q_k \neq 0}(p_k + q_k) \leq p+q+1$.
\end{example}

For the last example, we can determine exactly when 
 $SO(p,q)/\prod_{k=1}^{r}SO(p_k,q_k)$ has a discrete series
 as we will see from now.

First, consider the case where
 $G/H=SO(p,q)/\prod_{k=1}^{r}SO(p_k,q_k)$ is a symmetric space. 
Then by results of Flensted-Jensen~\cite{FJ} and
 Oshima-Matsuki~\cite{OM},
 $\textup{Disc}(G/H)\neq \emptyset$ if and only if 
 $\operatorname{rank} G/H = \operatorname{rank} K/(H\cap K)$,
 where $K=S(O(p)\times O(q))$.
More explicitly in terms of parameters,
 $SO(p,q)/\prod_{k=1}^{r}SO(p_k,q_k)$ is a symmetric space
 when $r=2$, $p_1+p_2=p$ and $q_1+q_2=q$.
Assume $p_1+q_1\geq p_2+q_2$.
Then
$\textup{Disc}(SO(p,q)/(SO(p_1,q_1)\times SO(p_2,q_2)))\neq \emptyset$
 if and only if $p_1\geq p_2$ and $q_1\geq q_2$.

Next, in the case where
 $G/H=SO(p,q)/\prod_{k=1}^{r}SO(p_k,q_k)$ is not a symmetric space,
 the exact condition is given by the following theorem.

\begin{theorem}\label{thm:SOpqDS}
Let $p \geq p_1 + \cdots + p_r$ and $q \geq q_1 + \cdots + q_r$.
Assume $p_1+q_1\geq p_2+q_2\geq \cdots \geq p_r+q_r$.
Set $(G,H)=(SO(p,q), \prod_{k=1}^{r}SO(p_k,q_k))$
 and assume $(G,H)$ is not a symmetric pair.
\begin{enumerate}
\item If $p$ and $q$ are even,
 then $\textup{Disc}(G/H)\neq \emptyset$.
\item Suppose that $p$ is even and $q$ is odd.
 Then $\textup{Disc}(G/H)\neq \emptyset$
 if and only if $2(p_1+q_1)\leq p+q$ or $q_1\neq 0$.
\item Suppose that $p$ and $q$ are odd.
 Then $\textup{Disc}(G/H)\neq \emptyset$
 if and only if $p_1q_1\neq 0$ and $2(p_1+q_1)\geq p+q+2$.
\end{enumerate}
\end{theorem}

The existence of discrete series,
 namely, `if' part of this theorem is a consequence of a result in \cite{HO}.
For the statement of this result, we introduce some notation.
Let $\mathfrak{g}^*$ denote the dual space of $\mathfrak{g}$
 and let $\mathfrak{h}^{\perp}:=\{\lambda\in\mathfrak{g}^*\mid \lambda|_{\mathfrak{h}}=0\}$.
An element $\lambda\in \mathfrak{g}^*$ is said to be elliptic if
 there exists a Cartan involution $\theta$ of $\mathfrak{g}$ such that
 $\lambda\in (\mathfrak{g}^*)^{\sigma}$.
Write $\mathfrak{g}^*_{\textup{ell}}$ for the set of all elliptic elements
 in $\mathfrak{g}^*$.

\begin{theorem}[{\cite[Theorem 1.6]{HO}}]\label{thm:exDS}
Let $G$ be a real reductive group and $H$ a reductive subgroup.
Then $\textup{Disc}(G/H)\neq \emptyset$ if $\mathfrak{h}^{\perp}\cap \mathfrak{g}^*_{\textup{ell}}$
 contains a non-empty open subset of $\mathfrak{h}^{\perp}$.
\end{theorem}

\begin{proof}[Proof of Theorem~\ref{thm:SOpqDS}]
For all the cases (1) -- (3), the given conditions on $p,q,p_1,q_1$ imply that
 $\mathfrak{h}^{\perp}\cap \mathfrak{g}^*_{\textup{ell}}$
 contains a non-empty open subset of $\mathfrak{h}^{\perp}$.
Then Theorem~\ref{thm:exDS} shows the `if' part of Theorem~\ref{thm:SOpqDS}.

Let us proof the `only if' part.

Suppose first that $p$ and $q$ are odd.
Moreover, suppose that $p_1q_1=0$ or $2(p_1+q_1)\leq p+q+1$. 
Then $L^2(G/H)$ is square integrable by  the assertion of Example~\ref{ex:SOpq}.
Since $\textup{Disc}(SO(p,q))=\emptyset$,
 Corollary~\ref{cor:DS} gives $\textup{Disc}(G/H)=\emptyset$.

Suppose next that $p$ is even and $q$ is odd.
Moreover, suppose that $q_1=0$ and $2p_1\geq p+q+1$. 
Then again by the assertion of Example~\ref{ex:SOpq}, $L^2(G/H)$ is square integrable.
Hence $\textup{Disc}(G/H)\subset \textup{Disc}(G)$ by Corollary~\ref{cor:DS}.
Let $\pi\in \textup{Disc}(G)$, that is, the Harish-Chandra's discrete series of $SO(p,q)$.
By calculating the possible $K$-types of $\pi\in \textup{Disc}(G)$,
 it turns out that $\pi$ does not have a nonzero $SO(p_1)$-fixed vector.
Therefore, $\pi$ cannot be realized on $G/H$.
As a consequence, $\textup{Disc}(G/H)=\emptyset$ in this case.
\end{proof}

\noindent
{\bf Acknowledgments}

Y.Oshima was supported by JSPS KAKENHI Grant Number JP24K06706.


\begin{thebibliography}{10}    
    \bibitem{Kaz}
    B. Bekka, P. de la Harpe and A. Valette, \textit{Kazhdan's Property (T)}. New Math. Monogr., \textbf{11}. Cambridge University Press, Cambridge, 2008.
    
    \bibitem{BK}
    Y. Benoist and T. Kobayashi, \textit{Tempered reductive homogeneous spaces}.
    J. Eur. Math. Soc. (JEMS) \textbf{17} (2015), no.~12, 3015--3036.
    
    \bibitem{BKII}
    Y. Benoist and T. Kobayashi, \textit{Tempered homogeneous spaces II}. 
    Dynamics, geometry, number theory --- the impact of Margulis on modern mathematics, 213--245, University of Chicago Press, Chicago, IL, 2022.
    
    \bibitem{BKIII}
    Y. Benoist and T. Kobayashi, \textit{Tempered homogeneous spaces III}. 
    J. Lie Theory \textbf{31} (2021), no.~3, 833--869. 
    
    \bibitem{BKIV}
    Y. Benoist and T. Kobayashi,
    \textit{Tempered homogeneous spaces IV}.
    J. Inst. Math. Jussieu \textbf{22} (2023), no.~6, 2879--2906.
        
    \bibitem{CHH}
    M. Cowling, U. Haagerup and R. Howe,
    \textit{Almost $L^2$ matrix coefficients}.
    J. Reine Angew. Math. \textbf{387} (1988), 97--110.

    \bibitem{FJ}
    M. Flensted-Jensen,
    \textit{Discrete series for semisimple symmetric spaces}.
    Ann. of Math. (2) \textbf{111} (1980), no.~2, 253--311.

    \bibitem{HC66}
    Harish-Chandra,
    \textit{Discrete series for semisimple {L}ie groups. {II}. {E}xplicit determination of the characters}. Acta Math.\ \textbf{116} (1966), 1--111.

    \bibitem{HO}
    B. Harris and Y. Oshima,
    \textit{On the asymptotic support of Plancherel measures for homogeneous spaces}. Duke Math. J. \textbf{173} (2024), no.~14, 2729--2807.

    \bibitem{HT}
    R. Howe and E. C. Tan,
    \textit{Nonabelian harmonic analysis: Applications of $\mathrm{SL}(2,\mathbb{R})$}. 
    Universitext, Springer-Verlag, New York, 1992.

    \bibitem{OM}
    T. Oshima and T. Matsuki,
    \textit{A description of discrete series for semisimple symmetric spaces}.
    Group representations and systems of differential equations ({T}okyo, 1982), 331--390, Adv. Stud. Pure Math., \textbf{4}, North-Holland, Amsterdam, 1984.

    
    \end{thebibliography}
\end{document}